\documentclass[11pt,letterpaper,reqno]{amsart} 

\usepackage{tikz}
\usetikzlibrary{positioning, shapes.geometric, arrows.meta}
\usepackage{amssymb}
\usepackage{amsmath}
\usepackage{amsthm}
\usepackage{amsfonts}
\usepackage{bbm}
\usepackage{enumitem}
\usepackage{pgfplots}
\pgfplotsset{compat=1.18}
\usepackage{booktabs}

\usepackage{graphicx}
\usepackage[T1]{fontenc}
\usepackage{doi}
\usepackage{float}
\usepackage{bookmark}
\usepackage{hyperref}
\hypersetup{pdfstartview={FitH}}

\usepackage{mathtools}

\numberwithin{equation}{section}

\newtheorem{thm}{Theorem}[section]
\newtheorem{lem}[thm]{Lemma}
\newtheorem{prop}[thm]{Proposition}
\newtheorem{cor}[thm]{Corollary}

\newtheorem{ques}[thm]{Question}
\newtheorem{conj}[thm]{Conjecture}

\theoremstyle{definition}

\newtheorem{rem}[thm]{Remark}

\DeclareMathOperator{\Tr}{Tr}
\DeclareMathOperator{\Ker}{Ker}

\DeclareMathOperator{\diag}{diag}

\newcommand{\C}{\mathbb{C}}
\newcommand{\mnc}{\mathbb{M}_n(\mathbb{C})}
\newcommand{\comm}[2]{\left[#1,#2\right]}

\begin{document}
	\title[Aluthge transforms and Hilbert--Schmidt self-commutators]%
	{On a conjecture of $\lambda$-Aluthge transforms and Hilbert--Schmidt self-commutators}
	
	\author[T.~Zhang]{Teng Zhang}
	\address{School of Mathematics and Statistics, Xi'an Jiaotong University, Xi'an 710049, P. R. China}
	\email{teng.zhang@stu.xjtu.edu.cn}
	
\subjclass[2020]{47B20, 15A60, 47A30, 47A63}
	\keywords{$\lambda$-Aluthge transform, self-commutator, Frobenius norm, Heinz-type deviation inequality}
	
	\begin{abstract}
Let $A$ be a complex square matrix, and write its polar decomposition as $A=U|A|$. For $0<\lambda<1$,
the $\lambda$-Aluthge transform of $A$ is defined by
\[
\Delta_\lambda(A)=|A|^\lambda U|A|^{1-\lambda}.
\]
In 2007, Huang and Tam  conjectured that the Frobenius norm of the self-commutator is contractive under
$\Delta_\lambda$: for every $0<\lambda<1$,
\[
\|\comm{A^*}{A}\|_{F}
\ \ge\
\|\comm{\Delta_\lambda(A)^*}{\Delta_\lambda(A)}\|_{F}.
\]
If this inequality held, then the iterated self-commutator norms
\[
\Bigl\{\bigl\|\comm{\Delta_\lambda^{\,m}(A)^*}{\Delta_\lambda^{\,m}(A)}\bigr\|_F\Bigr\}_{m\in\mathbb N}
\]
would form a nonincreasing sequence and necessarily converge to $0$.
In this paper we provide a counterexample, thereby disproving the conjecture.
We also obtain the quantitative bounds
\[
\sqrt{\frac32}\ \le\
\sup_{\substack{A\in\mnc,\ A^*A\neq AA^*\\ 0<\lambda<1}}
\frac{\|\comm{\Delta_\lambda(A)^*}{\Delta_\lambda(A)}\|_F}{\|\comm{A^*}{A}\|_F}
\ \le\ 2.
\]
	\end{abstract}
	
	\maketitle
	
	\section{Introduction}

	Let $\mnc$ be the set of all complex $n\times n$ matrices.
	For $A\in\mathbb M_n(\mathbb C)$, we use $A^*$ to denote its conjugate transpose.
	The trace of $A$ is defined by $\Tr A$, and the Frobenius (Hilbert--Schmidt) norm is
	\[
	\|A\|_F:=\sqrt{\Tr(A^*A)}.
	\]
	For $X,Y\in\mathbb M_n(\mathbb C)$, their commutator is
	\[
	[X,Y]:=XY-YX.
	\]
	In particular, the self-commutator of $A$ is
	\[
	[A^*,A]=A^*A-AA^*,
	\]
	which vanishes if and only if $A$ is normal.	For Hermitian matrices $X,Y\in\mathbb M_n(\mathbb C)$, we write
	$
	X\le Y
	$
	(in the Löwner order) if $Y-X$ is positive semidefinite.
	We also write $\diag(d_1,\dots,d_n)$ for the diagonal matrix with diagonal entries $d_1,\dots,d_n$.
	
	For every $A\in\mnc$, there exists a polar decomposition $A=UP$, where
	$P=|A|=(A^*A)^{1/2}\ge 0$ and $U$ is a partial isometry (unitary when $A$ is invertible).
	The (usual) Aluthge transform of $A$ \cite[p.~179]{AMS05} (see also \cite{Alu90}) is defined by
	\[
	\Delta(A)=P^{1/2}UP^{1/2}.
	\]
	Although the unitary factor in the polar decomposition may fail to be unique when $A$ is singular,
	the Aluthge transform is nevertheless well defined: it does not depend on the particular choice of
	the unitary factor. More generally, for $0<\lambda<1$ one defines the $\lambda$-Aluthge transform
	\cite[Def.~3.1, p.~181]{AMS05}
	\[
	\Delta_\lambda(A)=P^\lambda U P^{1-\lambda},
	\]
	so that $\Delta=\Delta_{1/2}$. Iterations of the Aluthge transform, and their convergence properties,
	have attracted considerable attention.
	
	Zhan \cite[Section~16]{Zha08} recorded two conjectures concerning the convergence of iterated Aluthge
	transforms in his remarkable problem collection, ``\emph{Open Problems in Matrix Theory}.''
	
	The first conjecture concerns the convergence of the Aluthge sequence.
	\begin{conj}[{\cite[Conjecture~20]{Zha08}}]\label{conj:aluthge-seq}
		Let $A\in \mathbb{M}_n(\mathbb{C})$. Then the sequence $\{\Delta^{m}(A)\}_{m=1}^\infty$ converges.
	\end{conj}
	
	Conjecture~\ref{conj:aluthge-seq} was originally conjectured by I.B. Jung, E. Ko and C. Pearcy in
	\cite{JKP00} in the operator setting on an infinite-dimensional Hilbert space. Cho, Jung and Lee
	\cite{CJL05} proved that the norm convergence of the Aluthge sequence may fail in the
	infinite-dimensional case. Ando and Yamazaki~\cite{AY03} verified
	Conjecture~\ref{conj:aluthge-seq} in the case where $A\in\mathbb{M}_2(\C)$; see~\cite{JKP03} for some
	special cases. Huang and Tam~\cite{HT07} proved that if the nonzero eigenvalues of $A$ have distinct
	moduli, then the $\lambda$-Aluthge sequence $\{\Delta_\lambda^{\,m}(A)\}_{m=1}^\infty$ converges.
	A major milestone toward Conjecture~\ref{conj:aluthge-seq} is the work of Antezana, Pujals and Stojanoff
	on the global dynamics of the Aluthge map. For diagonalizable matrices, they
	\cite{APS08} proved convergence of the iterated $\lambda$-Aluthge transform and
	gave a detailed description of the normal limit in terms of spectral data. In a subsequent paper
	\cite{APS11}, they established convergence of the (usual) Aluthge iterates for
	every complex matrix, providing a definitive finite-dimensional convergence theorem, thereby
	confirming Conjecture~\ref{conj:aluthge-seq}. Related ``Aluthge iterations'' have also been studied
	beyond the matrix and Hilbert space framework: Huang and Tam \cite{HT10} initiated an analysis
	of Aluthge iteration in connected noncompact semisimple Lie groups, and Tam and Thompson
	\cite{TT14} later obtained convergence results for semisimple Lie groups from a
	Lie-theoretic/geometric perspective. More recently, Antezana and Lim \cite{AL26} studied
	the $\lambda$-Aluthge transform on $\mathrm{GL}_n$ as a smooth dynamical system (showing it is a
	$C^\infty$ diffeomorphism for $\lambda\neq \tfrac12$) and analyzed iterations of its inverse map.
	
	It is also known that $\Delta(A)$ and $A$ have the same eigenvalues, and that whenever the Aluthge
	sequence converges, its limit is necessarily normal; see \cite[Proposition~1.10]{JKP00},
	\cite[Proposition~3.1]{JKP03}, or \cite[Theorem~1]{And04}. These observations support the view that
	the Aluthge transform behaves as a ``normalizing'' procedure. Moreover, for a matrix
	$A\in \mathbb{M}_n(\mathbb{C})$, the self-commutator $\comm{A^*}{A}$ vanishes if and only if $A$ is normal,
	so norms of $\comm{A^*}{A}$ provide natural quantitative measures of non-normality. A precise formulation
	of this heuristic is given by the following second conjecture (see also the Huang--Tam original paper
	\cite[Conjecture~3.3]{HT07}), which predicts that the Frobenius norm of the self-commutator is contractive under
	$\Delta_\lambda$.  Assuming the conjecture holds, \cite[Theorem~1.4]{HT07} implies that the sequence
	$\Bigl\{\bigl\|\comm{\Delta_\lambda^{\,m}(A)^*}{\Delta_\lambda^{\,m}(A)}\bigr\|_F\Bigr\}_{m\in\mathbb N}$ is nonincreasing and converges to $0$. To the best of our knowledge, there has been little progress on this conjecture over the past two decades.

	\begin{conj}[{\cite[Conjecture~21]{Zha08}}]\label{conj:selfcomm}
		Let $A\in \mathbb{M}_n(\mathbb{C})$. Then for any $0<\lambda<1$,
		\begin{equation}\label{eq:main}
			\|\comm{A^*}{A}\|_{F}
			\ \ge\
			\|\comm{\Delta_\lambda(A)^*}{\Delta_\lambda(A)}\|_{F}.
		\end{equation}
	\end{conj}
	
	Besides convergence and Frobenius self-commutator contraction, a number of works investigate which
	operator-theoretic or dynamical properties are preserved under Aluthge-type transforms and their
	generalizations. Tran \cite{Tra22} proved that for invertible operators on Hilbert space, dynamical
	notions such as bounded shadowing and quasi-hyperbolicity are invariant under $\lambda$-Aluthge
	transforms; see also Morales and Linh \cite{ML24} for further connections between
	shadowing, hyperbolicity and Aluthge transforms in Banach space dynamics. On the
	structure-preserver side, Botelho, Moln\'ar and Nagy \cite{BMN16} characterized
	bijective linear maps between von Neumann factors that commute with the $\lambda$-Aluthge transform.
	Induced and multivariable versions have also been investigated: Golla, Osaka, Udagawa and Yamazaki
	\cite{GOUY23} studied stability of the absolutely norm attaining (AN) property
	for families of induced Aluthge transformations defined via operator means, Osaka and Yamazaki
	\cite{OY25} analyzed limits of iterates of induced Aluthge transformations for centered
	operators, and Benhida, Curto, Lee and Yoon \cite{BCLY22} determined the spectral
	picture and joint spectral radius for a generalized spherical Aluthge transform of commuting
	$d$-tuples.

In this paper, we show that Conjecture~\ref{conj:selfcomm} fails already for the usual Aluthge transform ($\lambda=\tfrac12$).
We then propose the following two questions: 
\begin{ques}\label{ques:1}
Fix  $0<\lambda<1$. Determine the smallest constant $C_\lambda\ge 1$ such that for all $A\in \mathbb{M}_n(\mathbb{C})$,
	\begin{equation}\label{eq:mainC}
		C_\lambda\,
		\|\comm{A^*}{A}\|_{F}
		\ \ge\
		\|\comm{\Delta_\lambda(A)^*}{\Delta_\lambda(A)}\|_{F}.
	\end{equation}
\end{ques}
\begin{ques}\label{ques:2}
 Determine the smallest constant $C_*$, independent of $\lambda\in(0,1)$, such that for all $A\in \mathbb{M}_n(\mathbb{C})$ and all $0<\lambda<1$,
	\begin{equation}\label{eq:mainC-star}
		C_*\,
		\|\comm{A^*}{A}\|_{F}
		\ \ge\
		\|\comm{\Delta_\lambda(A)^*}{\Delta_\lambda(A)}\|_{F}.
	\end{equation}
\end{ques}

Clearly,
\[
C_*=\sup_{0<\lambda<1} C_\lambda.
\]

For Question~\ref{ques:1}, in the case of the usual Aluthge transform, we show that
\[
\sqrt{\frac{1+\sqrt2}{2}}\le C_{1/2}\le 2.
\]
For Question~\ref{ques:2}, we prove the two-sided estimate
\[
\sqrt{\frac32}\ \le\ C_*\ \le\ 2.
\]

\noindent\textbf{Methods and main ideas.}
The counterexample in Section~\ref{sec:counterexample} is based on choosing $A=UP$ with $U$ a cyclic permutation unitary and
$P$ diagonal, so that $A^*A=P^2$ and $AA^*=UP^2U^*$ are both diagonal and the Frobenius norms of the
self-commutators can be computed explicitly.
For the lower bounds in Section~\ref{sec:sharp} we work within a weighted cyclic shift family $A_{\varepsilon,s}=UP$,
derive closed-form expressions for $\|\,[A_{\varepsilon,s}^*,A_{\varepsilon,s}]\,\|_F$ and
$\|\,[\Delta_\lambda(A_{\varepsilon,s})^*,\Delta_\lambda(A_{\varepsilon,s})]\,\|_F$, and then optimize
asymptotically along suitable parameter regimes.
For the upper bound in Section~\ref{sec:upper}, we estimate the difference between $A^*A$ and
$\Delta_\lambda(A)^*\Delta_\lambda(A)$ (and similarly for $\Delta_\lambda(A)\Delta_\lambda(A)^*$)
through a Frobenius-norm Heinz-type deviation inequality proved using spectral decompositions and
Hilbert--Schmidt orthogonality, and then conclude by a triangle inequality argument.	

\medskip
\noindent\textbf{Organization of this paper.}
In Section~\ref{sec:counterexample} we present an explicit $4\times 4$ counterexample showing that
$\|\,[A^*,A]\,\|_F$ need not be contractive under the (usual) Aluthge transform.
In Section~\ref{sec:sharp} we introduce a two-parameter family of weighted cyclic shifts and use it to obtain sharp
lower bounds for the optimal constants $C_{1/2}$ and $C_*$ within this model family.
Finally, in Section~\ref{sec:upper} we prove a uniform upper bound $C_\lambda\le 2$ for all $0<\lambda<1$ via a
Heinz-type deviation inequality for the Frobenius norm.
	
	\section{A counterexample to Conjecture~\ref{conj:selfcomm} for $\lambda=\tfrac12$}\label{sec:counterexample}
	In this section, we provide a counterexample to Conjecture~\ref{conj:selfcomm} for $\lambda=\tfrac12$. 
	\begin{prop}\label{prop:counter}
		Let
		\[
		A=
		\begin{pmatrix}
			0&0&0&36\\
			1&0&0&0\\
			0&36&0&0\\
			0&0&49&0
		\end{pmatrix}\in \mathbb{M}_4(\mathbb{C}),
		\qquad \lambda=\tfrac12.
		\]
		Then $A$ is invertible and
		\[
		\|\comm{A^*}{A}\|_F
		\;<\;
		\|\comm{\Delta(A)^*}{\Delta(A)}\|_F,
		\]
		so the inequality in \eqref{eq:main} fails (already for the usual Aluthge transform).
	\end{prop}
	
	\begin{proof}
		Set
		\[
		U=
		\begin{pmatrix}
			0&0&0&1\\
			1&0&0&0\\
			0&1&0&0\\
			0&0&1&0
		\end{pmatrix},
		\qquad
		P=\diag(1,36,49,36).
		\]
		Then $U$ is unitary and one checks that $A=UP$. Hence $A$ is invertible and
	$|A|=P$.
		Therefore the usual Aluthge transform is
		\[
		\Delta(A)=|A|^{1/2}U|A|^{1/2}=P^{1/2}UP^{1/2},
		\qquad
		P^{1/2}=\diag(1,6,7,6).
		\]
		
		We first compute the self-commutator of $A$. Since
		\[
		A^*A=P^2=\diag(1,1296,2401,1296),
		\qquad
		AA^*=UP^2U^*=\diag(1296,1,1296,2401),
		\]
		it follows that
		\[
		\comm{A^*}{A}
		=
		\diag(-1295,\,1295,\,1105,\,-1105),
		\]
		and hence
		\[
		\|\comm{A^*}{A}\|_F^2
		=
		2\cdot1295^2+2\cdot1105^2
		=
		5{,}796{,}100.
		\]
		
		Next we compute the self-commutator of $\Delta(A)$. We have
		\[
		\Delta(A)^*\Delta(A)=P^{1/2}U^*PU P^{1/2}
		=\diag(36,1764,1764,36),
		\]
		and
		\[
		\Delta(A)\Delta(A)^*=P^{1/2}UPU^*P^{1/2}
		=\diag(36,36,1764,1764).
		\]
		Therefore
		\[
		\comm{\Delta(A)^*}{\Delta(A)}
		=
		\diag(0,\,1728,\,0,\,-1728),
		\]
		so
		\[
		\|\comm{\Delta(A)^*}{\Delta(A)}\|_F^2
		=
		2\cdot 1728^2
		=
		5{,}971{,}968.
		\]
		Since $5{,}971{,}968>5{,}796{,}100$, we obtain
		\[
		\|\comm{\Delta(A)^*}{\Delta(A)}\|_F
		>
		\|\comm{A^*}{A}\|_F,
		\]
		which proves the claim.
	\end{proof}

	\section{Sharp lower bounds for the optimal constants}\label{sec:sharp}
	We will use the following cyclic shift model.
	\begin{prop}\label{prop:family}
		Let $\varepsilon>0$ and $s>0$, and define
		\[
		U=
		\begin{pmatrix}
			0&0&0&1\\
			1&0&0&0\\
			0&1&0&0\\
			0&0&1&0
		\end{pmatrix},
		\qquad
		P=\diag(\varepsilon,1,s,1),
		\qquad
		A_{\varepsilon,s}:=UP\in\mathbb M_4(\C).
		\]
		Then $A_{\varepsilon,s}$ is invertible and $|A_{\varepsilon,s}|=P$.
		Moreover,
		\begin{align}
			\|A_{\varepsilon,s}^*A_{\varepsilon,s}-A_{\varepsilon,s}A_{\varepsilon,s}^*\|_F^2
			&=
			2(1-\varepsilon^2)^2+2(s^2-1)^2, \label{eq:den}\\[2mm]
			\|\Delta_\lambda(A_{\varepsilon,s})^*\Delta_\lambda(A_{\varepsilon,s})
			-\Delta_\lambda(A_{\varepsilon,s})\Delta_\lambda(A_{\varepsilon,s})^*\|_F^2
			&=
			(\varepsilon^{2-2\lambda}-\varepsilon^{2\lambda})^2+(s^{2\lambda}-\varepsilon^{2-2\lambda})^2 \nonumber\\
			&\quad+(s^{2-2\lambda}-s^{2\lambda})^2+(\varepsilon^{2\lambda}-s^{2-2\lambda})^2. \label{eq:num}
		\end{align}
	\end{prop}
	
	\begin{proof}
		Since $U$ is unitary and $P>0$ is diagonal, $A_{\varepsilon,s}=UP$ is invertible and
		$A_{\varepsilon,s}^*A_{\varepsilon,s}=P^2$, hence $|A_{\varepsilon,s}|=(A_{\varepsilon,s}^*A_{\varepsilon,s})^{1/2}=P$.
		The identity \eqref{eq:den} follows from
		$A_{\varepsilon,s}^*A_{\varepsilon,s}=P^2$ and $A_{\varepsilon,s}A_{\varepsilon,s}^*=UP^2U^*$
		(which permutes the diagonal entries cyclically).
		
		For \eqref{eq:num}, note that
		\[
		\Delta_\lambda(A_{\varepsilon,s})=P^\lambda U P^{1-\lambda}
		=
		\begin{pmatrix}
			0&0&0&\varepsilon^\lambda\\
			\varepsilon^{1-\lambda}&0&0&0\\
			0&s^\lambda&0&0\\
			0&0&s^{1-\lambda}&0
		\end{pmatrix}.
		\]
		This is a weighted cyclic shift. A direct calculation gives that both
		$\Delta_\lambda(A_{\varepsilon,s})^*\Delta_\lambda(A_{\varepsilon,s})$ and
		$\Delta_\lambda(A_{\varepsilon,s})\Delta_\lambda(A_{\varepsilon,s})^*$ are diagonal, and their diagonal
		entries are consecutive cyclic shifts of the squared weights
		$\varepsilon^{2-2\lambda}, s^{2\lambda}, s^{2-2\lambda}, \varepsilon^{2\lambda}$, yielding \eqref{eq:num}.
	\end{proof}
	Next, we give  a sharp lower bound for $C_{1/2}$.
	\begin{thm}\label{thm:half}
		For $\lambda=\tfrac12$ in \eqref{eq:mainC},
		\[
		C_{1/2}
		\ \ge\
		\sqrt{\frac{1+\sqrt2}{2}}.
		\]
Moreover, within the family $A_{\varepsilon,s}$ of Proposition~\ref{prop:family},
\[
\sup_{s>0}\ \lim_{\varepsilon\to 0^+}
\frac{\|\Delta(A_{\varepsilon,s})^*\Delta(A_{\varepsilon,s})-\Delta(A_{\varepsilon,s})\Delta(A_{\varepsilon,s})^*\|_F}
{\|A_{\varepsilon,s}^*A_{\varepsilon,s}-A_{\varepsilon,s}A_{\varepsilon,s}^*\|_F}
=
\sqrt{\frac{1+\sqrt2}{2}},
\]
and this value is approached by taking $s=2^{1/4}$ and letting $\varepsilon\to 0^+$.
	\end{thm}
	
	\begin{proof}
		When $\lambda=\tfrac12$, \eqref{eq:num} simplifies to
		\[
		\|\Delta(A_{\varepsilon,s})^*\Delta(A_{\varepsilon,s})-\Delta(A_{\varepsilon,s})\Delta(A_{\varepsilon,s})^*\|_F^2
		=
		2(s-\varepsilon)^2,
		\]
		so the ratio of squares equals
		\[
		R(\varepsilon,s)^2
		:=
		\frac{\|\Delta(A_{\varepsilon,s})^*\Delta(A_{\varepsilon,s})-\Delta(A_{\varepsilon,s})\Delta(A_{\varepsilon,s})^*\|_F^2}
		{\|A_{\varepsilon,s}^*A_{\varepsilon,s}-A_{\varepsilon,s}A_{\varepsilon,s}^*\|_F^2}
		=
		\frac{(s-\varepsilon)^2}{(1-\varepsilon^2)^2+(s^2-1)^2}.
		\]
		Letting $\varepsilon\to 0^+$ yields
		\[
		\lim_{\varepsilon\to 0^+}R(\varepsilon,s)^2
		=
		\frac{s^2}{1+(s^2-1)^2}.
		\]
		Writing $x=s^2>0$, we maximize
		\[
		\phi(x)=\frac{x}{1+(x-1)^2}=\frac{x}{x^2-2x+2}.
		\]
		A derivative calculation gives
		$\phi'(x)=\frac{-x^2+2}{(x^2-2x+2)^2}$, hence $\phi$ attains its maximum at $x=\sqrt2$ and
		\[
		\max_{x>0}\phi(x)=\phi(\sqrt2)=\frac{1+\sqrt2}{2}.
		\]
	Thus, within this family the ratio can be made arbitrarily close to
	$\sqrt{\frac{1+\sqrt2}{2}}$, and therefore $C_{1/2}\ge \sqrt{\frac{1+\sqrt2}{2}}$.
	\end{proof}
	Now, we present an uniform sharp lower bound on $C_*$.
	\begin{thm}\label{thm:uniform}
		The uniform optimal constant in \eqref{eq:mainC-star} satisfies
		\[
		C_*
		\ \ge\
		\sqrt{\frac32}.
		\]
Moreover, within the family $A_{\varepsilon,s}$ of Proposition~\ref{prop:family},
\[
\sup_{\substack{s>0\\0<\lambda<1}}\ \lim_{\varepsilon\to 0^+}
\frac{\|\Delta_\lambda(A_{\varepsilon,s})^*\Delta_\lambda(A_{\varepsilon,s})-\Delta_\lambda(A_{\varepsilon,s})\Delta_\lambda(A_{\varepsilon,s})^*\|_F}
{\|A_{\varepsilon,s}^*A_{\varepsilon,s}-A_{\varepsilon,s}A_{\varepsilon,s}^*\|_F}
=
\sqrt{\frac32},
\]
and this value is approached by taking $s=\sqrt2$, letting $\lambda\to 0^+$ (or $\lambda\to 1^-$),
and letting $\varepsilon\to 0^+$.
	\end{thm}
	
	\begin{proof}
		Fix $0<\lambda<1$ and set $x=s^2>0$. Letting $\varepsilon\to 0^+$ in \eqref{eq:num} gives
		\[
		\lim_{\varepsilon\to 0^+}
		\|\Delta_\lambda(A_{\varepsilon,s})^*\Delta_\lambda(A_{\varepsilon,s})-\Delta_\lambda(A_{\varepsilon,s})\Delta_\lambda(A_{\varepsilon,s})^*\|_F^2
		=
		2\bigl(x^{2\lambda}+x^{2-2\lambda}-x\bigr),
		\]
		while \eqref{eq:den} yields
		\[
		\lim_{\varepsilon\to 0^+}
		\|A_{\varepsilon,s}^*A_{\varepsilon,s}-A_{\varepsilon,s}A_{\varepsilon,s}^*\|_F^2
		=
		2(x^2-2x+2).
		\]
		Hence along $\varepsilon\to 0^+$ the ratio of squares converges to
		\[
		F_\lambda(x):=\frac{x^{2\lambda}+x^{2-2\lambda}-x}{x^2-2x+2}.
		\]
		For fixed $x>0$, the function $g(\lambda):=x^{2\lambda}+x^{2-2\lambda}$ is convex in $\lambda$ on $[0,1]$
		(since it is a sum of exponentials), so its maximum over $[0,1]$ is attained at $\lambda=0$ or $1$:
		\[
		x^{2\lambda}+x^{2-2\lambda}\le 1+x^2.
		\]
		Therefore
		\[
		F_\lambda(x)\le \frac{x^2-x+1}{x^2-2x+2}=:H(x).
		\]
		We now maximize $H(x)$ over $x>0$. A direct derivative computation shows that $H$ attains its maximum
		at $x=2$, with
		\[
		\max_{x>0}H(x)=H(2)=\frac{3}{2}.
		\]
		Consequently,
		\[
		\sup_{0<\lambda<1}\sup_{x>0}F_\lambda(x)=\frac32,
		\]
		and hence, within the family $A_{\varepsilon,s}$,
		the ratio of norms can be made arbitrarily close to $\sqrt{3/2}$ by taking $x=2$ (i.e.\ $s=\sqrt2$),
		letting $\lambda\to 0^+$ (or $1^-$), and letting $\varepsilon\to 0^+$.
This proves the stated sharp lower bound for $C_*$.
	\end{proof}
	
	\begin{rem}\label{rem:status}
		Theorems~\ref{thm:half} and~\ref{thm:uniform} identify the sharp lower bounds for the optimal constants
		$C_{1/2}$ and $C_*$ within the explicit cyclic shift family
		$A_{\varepsilon,s}$. In particular, any constant in \eqref{eq:mainC} valid for all matrices must satisfy
		$C_{1/2}\ge \sqrt{\frac{1+\sqrt2}{2}}$, and any  uniform constant in \eqref{eq:mainC-star} must satisfy
		$C_*\ge \sqrt{\frac32}$. Determining whether these lower bounds are the true global optima over all
		matrices remains open.
	\end{rem}

	\section{A uniform upper bound via a Heinz-type inequality}\label{sec:upper}
	
	In this section, we record a simple uniform upper bound for the optimal constants.
	The key input is a Heinz-type deviation estimate for the Frobenius norm. Such inequalities are
	reminiscent of Heinz-type bounds in matrix analysis; see, for instance, \cite{Kit07} for related
	commutator inequalities.
	
	\begin{lem}[Heinz-type deviation inequality]\label{lem:heinz-dev}
		Let $X,Y\in \mnc$ be positive semidefinite and let $0\le t\le 1$.
		Then
		\begin{equation}\label{eq:heinz-dev}
			\Bigl\|\,X^{\frac{1-t}{2}}\,Y^{t}\,X^{\frac{1-t}{2}}-X\,\Bigr\|_{F}
			\ \le\
			\|Y-X\|_{F}.
		\end{equation}
	\end{lem}
	
	\begin{proof}
		The cases $t=0$ and $t=1$ are immediate. Assume $0<t<1$.
		
		\smallskip
		\noindent\textbf{Step 1: reduction to a one-sided term.}
		Set $A:=Y^{t}-X^{t}$ (so $A=A^*$) and $B:=X^{1-t}\ge 0$. Then
		\[
		X^{\frac{1-t}{2}}\,Y^{t}\,X^{\frac{1-t}{2}}-X
		=
		X^{\frac{1-t}{2}}(Y^{t}-X^{t})X^{\frac{1-t}{2}}
		=
		B^{1/2}AB^{1/2}.
		\]
		Using the Hilbert--Schmidt inner product $\langle S,T\rangle=\Tr(T^*S)$ and Cauchy--Schwarz inequality, we have
		\begin{align*}
			\|B^{1/2}AB^{1/2}\|_F^2
			&=\Tr\,\bigl((B^{1/2}AB^{1/2})^*(B^{1/2}AB^{1/2})\bigr)\\
			&=\Tr\,\bigl((BA)^*(AB)\bigr)\\
			&\le \|BA\|_F\,\|AB\|_F.
		\end{align*}
		Since $A=A^*$, we also have $\|AB\|_F=\|BA\|_F$, because
		\[
		\|AB\|_F^2=\Tr\bigl((AB)^*(AB)\bigr)=\Tr(BA^2B)=\Tr(A^2B^2)
		=\Tr(AB^2A)=\|BA\|_F^2.
		\]
		Therefore,
		\[
		\|B^{1/2}AB^{1/2}\|_F\le \|AB\|_F.
		\]
		Noting that $AB=(Y^t-X^t)X^{1-t}=Y^tX^{1-t}-X$, we obtain
		\begin{equation}\label{eq:heinz-step1}
			\bigl\|X^{\frac{1-t}{2}}Y^{t}X^{\frac{1-t}{2}}-X\bigr\|_F
			\le
			\|Y^{t}X^{1-t}-X\|_F.
		\end{equation}
		
		\smallskip
		\noindent\textbf{Step 2: a spectral decomposition argument (positive definite case).}
		Assume first that $X$ and $Y$ are positive definite. Let
		\[
		X=\sum_{j=1}^{m}\mu_j P_j,\qquad
		Y=\sum_{i=1}^{k}\lambda_i Q_i
		\]
		be spectral decompositions, where $\mu_j>0$, $\lambda_i>0$, and $\{P_j\}$, $\{Q_i\}$ are families of
		mutually orthogonal projections with $\sum_j P_j=\sum_i Q_i=I$. Set $H:=Y-X$ (so $H=H^*$).
		Then for each $i,j$,
		\begin{equation}\label{eq:heinz-QHP}
			Q_iHP_j=Q_iYP_j-Q_iXP_j=(\lambda_i-\mu_j)Q_iP_j.
		\end{equation}
		Moreover,
		\[
		Y^tX^{1-t}-X
		=\sum_{i,j}\bigl(\lambda_i^t\mu_j^{1-t}-\mu_j\bigr)\,Q_iP_j.
		\]
		Introduce the scalar function
		\[
		h_t(r):=
		\begin{cases}
			\dfrac{r^t-1}{r-1}, & r\neq 1,\\[6pt]
			t, & r=1,
		\end{cases}
		\qquad (r>0).
		\]
		Using $\lambda_i^t\mu_j^{1-t}-\mu_j=(\lambda_i-\mu_j)\,h_t(\lambda_i/\mu_j)$ and \eqref{eq:heinz-QHP}, we can write
		\begin{equation}\label{eq:heinz-multiplier}
			Y^tX^{1-t}-X=\sum_{i,j} h_t\!\Bigl(\frac{\lambda_i}{\mu_j}\Bigr)\,Q_iHP_j.
		\end{equation}
		We claim that
		\begin{equation}\label{eq:heinz-ht-bound}
			0\le h_t(r)\le 1\qquad (r>0,\ 0\le t\le 1).
		\end{equation}
		Indeed, $r^t$ is increasing in $r$, hence $r^t-1$ and $r-1$ have the same sign, so $h_t(r)\ge 0$.
		If $r\ge 1$, then $t\le 1$ implies $r^t\le r$, hence $h_t(r)=(r^t-1)/(r-1)\le 1$.
		If $0<r\le 1$, then $t\le 1$ implies $r^t\ge r$, hence
		$h_t(r)=(1-r^t)/(1-r)\le 1$. Finally, $h_t(1)=t\in[0,1]$.
		
		Next we use Hilbert--Schmidt orthogonality. The family $\{Q_iHP_j\}_{i,j}$ is pairwise orthogonal:
		if $(i,j)\neq(i',j')$, then
		\[
\Tr\!\bigl((Q_iHP_j)^*(Q_{i'}HP_{j'})\bigr)
=\Tr\!\bigl(P_jH Q_iQ_{i'}H P_{j'}\bigr)=0.
		\]
	Indeed, if $i\neq i'$ then $Q_iQ_{i'}=0$; while if $i=i'$ and $j\neq j'$, then by cyclicity of the trace,
	\[
	\Tr\!\bigl(P_jH Q_iH P_{j'}\bigr)=\Tr\!\bigl(P_{j'}P_jH Q_iH\bigr)=0
	\]
	since $P_{j'}P_j=0$.
		Moreover, since $(\sum_i Q_i)H(\sum_j P_j)=H$, we have the orthogonal decomposition
		$H=\sum_{i,j}Q_iHP_j$, and therefore
		\begin{equation}\label{eq:heinz-H-decomp}
			\|H\|_F^2=\sum_{i,j}\|Q_iHP_j\|_F^2.
		\end{equation}
		Combining \eqref{eq:heinz-multiplier}, orthogonality, \eqref{eq:heinz-ht-bound}, and \eqref{eq:heinz-H-decomp}, we obtain
		\[
		\|Y^tX^{1-t}-X\|_F^2
		=
		\sum_{i,j}\Bigl|h_t\!\Bigl(\frac{\lambda_i}{\mu_j}\Bigr)\Bigr|^2\,\|Q_iHP_j\|_F^2
		\le
		\sum_{i,j}\|Q_iHP_j\|_F^2
		=\|H\|_F^2
		=\|Y-X\|_F^2.
		\]
		Hence
		\begin{equation}\label{eq:heinz-step2}
			\|Y^tX^{1-t}-X\|_F\le \|Y-X\|_F
			\qquad\text{when }X,Y>0.
		\end{equation}
		
		\smallskip
		\noindent\textbf{Step 3: the positive semidefinite case.}
		For $\varepsilon>0$ set $X_\varepsilon:=X+\varepsilon I$ and $Y_\varepsilon:=Y+\varepsilon I$.
		Then $X_\varepsilon,Y_\varepsilon>0$, so by \eqref{eq:heinz-step1}--\eqref{eq:heinz-step2} applied to $(X_\varepsilon,Y_\varepsilon)$ we have
		\[
		\bigl\|X_\varepsilon^{\frac{1-t}{2}}Y_\varepsilon^{t}X_\varepsilon^{\frac{1-t}{2}}-X_\varepsilon\bigr\|_F
		\le
		\|Y_\varepsilon-X_\varepsilon\|_F
		=\|Y-X\|_F.
		\]
		Finally, in finite dimension the functional calculus map $Z\mapsto Z^\alpha$ is continuous on the cone
		of positive semidefinite matrices for each $\alpha\in[0,1]$. Hence letting $\varepsilon\downarrow 0$ yields
		\eqref{eq:heinz-dev} for general $X,Y\ge 0$.
	\end{proof}
	
	\begin{thm}[Uniform upper bound]\label{thm:upper}
		For any $A\in \mnc$ and any $0<\lambda<1$,
		\begin{equation}\label{eq:upper-main}
			\|\comm{\Delta_\lambda(A)^*}{\Delta_\lambda(A)}\|_F
			\ \le\
			2\,\|\comm{A^*}{A}\|_F.
		\end{equation}
		In particular,
		\[
		C_\lambda\le 2\qquad (0<\lambda<1),
		\qquad
		C_*\le 2.
		\]
	\end{thm}
	
	\begin{proof}
		Let $A=U|A|$ be a polar decomposition. In finite dimension, we may  choose $U$ unitary:
		start with the canonical partial isometry in the polar decomposition and extend it arbitrarily to a
		unitary on $\Ker(|A|)$ (this does not change $A=U|A|$ since $|A|$ vanishes on $\Ker(|A|)$, nor does it
		affect $\Delta_\lambda(A)$ because $|A|^{1-\lambda}$ vanishes on $\Ker(|A|)$).
		
		Set
		\[
		H:=A^*A=|A|^2,\qquad K:=AA^*=UHU^*,
		\]
		so that $\|H-K\|_F=\|\comm{A^*}{A}\|_F$.
		A direct computation gives
		\begin{equation}\label{eq:upper-expand}
			\Delta_\lambda(A)^*\Delta_\lambda(A)=|A|^{1-\lambda}U^*|A|^{2\lambda}U|A|^{1-\lambda},
			\qquad
			\Delta_\lambda(A)\Delta_\lambda(A)^*=|A|^\lambda U|A|^{2-2\lambda}U^*|A|^\lambda.
		\end{equation}
		
		We first estimate $\Delta_\lambda(A)^*\Delta_\lambda(A)-A^*A$.
		Conjugating by $U$ and using $|A|^{2\lambda}=H^\lambda$ and $U|A|U^*=(U|A|^2U^*)^{1/2}=K^{1/2}$, we obtain
		\[
		U\bigl(\Delta_\lambda(A)^*\Delta_\lambda(A)-A^*A\bigr)U^*
		=
		K^{\frac{1-\lambda}{2}}\,H^\lambda\,K^{\frac{1-\lambda}{2}}-K.
		\]
		Hence, by unitary invariance of $\|\cdot\|_F$ and Lemma~\ref{lem:heinz-dev} (with $X=K$, $Y=H$,
		$t=\lambda$),
		\begin{equation}\label{eq:upper-step1}
			\|\Delta_\lambda(A)^*\Delta_\lambda(A)-A^*A\|_F
			\le
			\|H-K\|_F.
		\end{equation}
		
		Next, we estimate $A^*A-\Delta_\lambda(A)\Delta_\lambda(A)^*$.
		From \eqref{eq:upper-expand} we have
		\[
		\Delta_\lambda(A)\Delta_\lambda(A)^*
		=
		H^{\lambda/2}\,(U H^{1-\lambda}U^*)\,H^{\lambda/2}
		=
		H^{\lambda/2}\,K^{1-\lambda}\,H^{\lambda/2},
		\]
		and Lemma~\ref{lem:heinz-dev} (with $X=H$, $Y=K$, $t=1-\lambda$) yields
		\begin{equation}\label{eq:upper-step2}
			\|A^*A-\Delta_\lambda(A)\Delta_\lambda(A)^*\|_F
			=
			\|H^{\lambda/2}K^{1-\lambda}H^{\lambda/2}-H\|_F
			\le
			\|H-K\|_F.
		\end{equation}
		
		Finally, by the triangle inequality and \eqref{eq:upper-step1}--\eqref{eq:upper-step2},
		\[
		\|\comm{\Delta_\lambda(A)^*}{\Delta_\lambda(A)}\|_F
		\le
		\|\Delta_\lambda(A)^*\Delta_\lambda(A)-A^*A\|_F
		+\|A^*A-\Delta_\lambda(A)\Delta_\lambda(A)^*\|_F
		\le 2\|H-K\|_F,
		\]
		which is exactly \eqref{eq:upper-main}.
	\end{proof}
	
	\begin{cor}\label{cor:upper}
		Combining Theorem~\ref{thm:uniform} with Theorem~\ref{thm:upper} yields the quantitative bounds
		\[
		\sqrt{\frac32}\ \le\ C_*\ \le\ 2.
		\]
	\end{cor}

	\section*{Declaration of competing interest}
	The author declares no competing interests.
	
	\section*{Data availability}
	No data was used for the research described in the article.
	
	\section*{Acknowledgments}
	The author is grateful to his advisor Minghua Lin for suggesting this problem. This work is supported by the China Scholarship Council, the Young Elite Scientists Sponsorship
	Program for PhD Students (China Association for Science and Technology), and the Fundamental Research
	Funds for the Central Universities at Xi'an Jiaotong University (Grant No.~xzy022024045).
	

\end{document}